\documentclass[a4paper,11pt,reqno,noindent]{amsart}
\usepackage[centertags]{amsmath}
\usepackage{amsfonts,amssymb,amsthm} 
\usepackage{graphicx} 

\usepackage[english]{babel}
\usepackage{newlfont}
\usepackage{color}
\usepackage[body={15cm,21.5cm},centering]{geometry} 
\usepackage{fancyhdr}
\usepackage{esint}
\usepackage{enumerate}

\usepackage[active]{srcltx}

\newtheorem{theorem}{Theorem}
\newtheorem{lemma}[theorem]{Lemma}
\newtheorem{proposition}[theorem]{Proposition}
\newtheorem{corollary}[theorem]{Corollary}
\newtheorem{remark}[theorem]{Remark}

\theoremstyle{definition}

\newcommand{\osc}[1]{{\underset{#1}{\mathrm{osc}}}}

\newcommand{\supp}{\operatorname{supp}}

\newcommand{\R}{\mathbb{R}}

\renewcommand{\d}{\mathrm{d}}
\renewcommand{\L}{\mathcal{L}}
\newcommand{\A}{\mathcal{A}}

\newcommand{\abs}[1]{\lvert#1\rvert}
\newcommand{\norm}[1]{\lvert\lvert#1\rvert\rvert}

\title{Symmetry breaking in indented elastic cones}
\date{December 22, 2015}
\author[S. Conti] {Sergio Conti}
\address[Sergio Conti]{Institut f\"ur Angewandte Mathematik, Universit\"at Bonn, Germany}
\email{sergio.conti@uni-bonn.de}
\author[H. Olbermann] {Heiner Olbermann}
\address[Heiner Olbermann]{Hausdorff Center for Mathematics, Bonn, Germany}
\email{heiner.olbermann@hcm.uni-bonn.de}
\author[I. Tobasco] {Ian Tobasco}
\address[Ian Tobasco]{Courant Institute of Mathematical Sciences, New York, USA}
\email{tobasco@cims.nyu.edu}

\begin{document}
\maketitle

\begin{abstract}
Motivated by simulations of carbon nanocones (see Jordan and  Crespi, Phys. Rev. Lett., 2004), 
we consider  a variational plate model for an elastic cone under
compression in the direction of the cone symmetry axis. Assuming radial
symmetry, and modeling the compression by suitable Dirichlet boundary
conditions at the center and the boundary of the sheet,  we identify the  energy scaling law in the von-K\'arm\'an plate
model. Specifically, we find that three different regimes
arise with increasing indentation $\delta$: initially the energetic cost
of the logarithmic singularity dominates, then there is a linear response corresponding
to a moderate deformation close to the boundary of the cone, and for larger $\delta$ a localized
inversion takes place in the central region.
Then we show that for large enough indentations  minimizers
 of the elastic energy cannot be radially symmetric. We do so by an explicit construction that achieves
lower elastic energy than the minimum amount possible for radially symmetric deformations. 
\end{abstract}

\section{Introduction}
Graphene, a single layer of carbon atoms  with a honeycomb two-dimensional lattice structure,
exhibits  unique properties and provides challenges in both experimental and fundamental sciences \cite{dresselhaus1997future,geim2009graphene,guo2011graphene}.
Recently, carbon nanocones 
with a defect in the lattice structure which generates a point singularity in the intrinsic curvature
have been produced.
They can be thought of as the result of removing  a wedge from a graphene sheet \cite{GE1994192,krishnan1997graphitic,Shenderova2001}. 
Atomic-scale simulations of indentation tests on carbon nanocones have shown a variety of response regimes with
increasing 
compression in the direction of the symmetry axis of the cone  \cite{PhysRevLett.93.255504}. In particular,
the simulation has shown that  when pressing the cone at its center with a small spherical
indenter, the configurations with (almost) minimal energy were radially symmetric.
We investigate these deformations from the viewpoint of the theory of elasticity. 
 As is true for graphene sheets and carbon nanotubes,
the qualitative behavior of the nanocone can be described by continuum elastic plate
models \cite{lee2008measurement,lu2009elastic}, although deviations have been observed in fine-scale experiments
where the  atomic structure becomes relevant \cite{yakobson1996nanomechanics}. In this paper, we focus on the von K\'arm\'an model,
which is appropriate for cones with small angular deficit and is expected to provide a qualitatively correct description also at angular deficit
of order 1.

The study of elastic deformation of thin elastic objects is an important field in mechanics and applied mathematics. 
On the one hand, remarkable successes have been obtained in the  derivation of reduced models from three-dimensional elasticity 
in the limit of small film thickness
\cite{ciarlet1980justification,CiarletII,FJM2002,FJM2006,LedretRaoult1995}. 
On the other hand, many insights on the behavior of compressed sheets have been obtained from a qualitative analysis of the variational formulation of elasticity.
The possibility of crumpling shows that models which properly reflect the invariance under rigid-body rotations are nonconvex, and
in many physical problems one explores a regime in which fine structures appear. 
An understanding of the qualitative properties of low-energy states can often be obtained by determining the scaling
of the minimum energy with respect to the relevant parameters, e.g.\ the film thickness or the magnitude of applied forcing. 
This approach started with the work by Kohn and M\"uller on shape memory alloys  \cite{KM94,KM92}, and was then generalized to
other mechanical problems, such as micromagnetism \cite{CKO99} and flux penetration patterns in superconductors \cite{CKO04,CCKO08}.
For the case of flat elastic sheets, it is known that uniform compression with Dirichlet boundary conditions leads
to an energy per unit volume proportional to the film thickness, $h$, with fold patterns which refine close to the boundary \cite{BCDM00,BCDM02,JinSternberg2}.
If instead one only imposes confinement to a small volume, without boundary conditions, it is known that the minimum energy is bounded from above
by a constant multiple of $h^{5/3}$ \cite{Maggi08,Venkataramani04,RevModPhys.79.643}. For conically constrained sheets, the minimal energy is proportional to 
$h^2\log 1/h$ \cite{BrandmanKohn2013,MuellerOlbermann2014}.

In this paper, we study the energetics of an elastic cone within the von K\'arm\'an plate model. Assuming radial 
symmetry, we show that three different regimes arise, with energy scaling  as $h^2\log 1/h$, $\delta^2 h^{1/2}$, or $\delta^{1/2}h^{3/2}$
respectively, where $h$ is the film thickness and $\delta$ is the indentation
(see Theorem \ref{thm:main} for the precise statement). The first expression is the one characteristic of the undisturbed cone, the second corresponds to a
moderate deformation close to the outer boundary of the cone, and the third one to an inversion in the central part of the cone,
see Section \ref{secUB} for a more detailed discussion.
We also prove that, if the radial symmetry assumption is removed, the energy required for large indentations decreases; specifically, we give an upper bound for the minimum energy that is proportional to $h^{5/3}$ (see Theorem
\ref{thm:ridgeconst} for the precise statement). The corresponding patterns are regularizations of an
``inverted pyramid'', see Figure \ref{fig:pyramid}.
We use a square symmetry in our explicit construction for simplicity, but expect that  a similar construction would be 
doable with minor changes for any regular polygon, and would also result in an energy scaling as $h^{5/3}$.
Such constructions are quite similar to the shape one observes when indenting a cone made out of paper, see Figure \ref{fig:indpapercone}, and  are  reminiscent of experimental observations in the buckling of spheres, see for example \cite{ruan2006crushing}.

The resulting picture is that, with increasing $\delta$, the indented cone undergoes several phase transformations. 
For $0\le \delta\lesssim h^{2/3}$ the deformation is localized around the outer boundary, and the response is essentially linear. 
In the regime $h^{2/3}\lesssim\delta\lesssim h^{1/3}$, buckling in the central region occurs, and
the cylindrical symmetry is preserved (``conical inversion''). For $\delta\gtrsim h^{1/3}$, the cylindrical symmetry is broken, and 
point singularities arise (``pyramidal inversion''). We remark, however, that our lower bounds only hold in the
restricted geometry with cylindrical symmetry. The upper bounds have no such restriction and can, indeed, be extended
to fully nonlinear models without major modifications.

\begin{figure}
\begin{center}
\includegraphics[width=.5\textwidth]{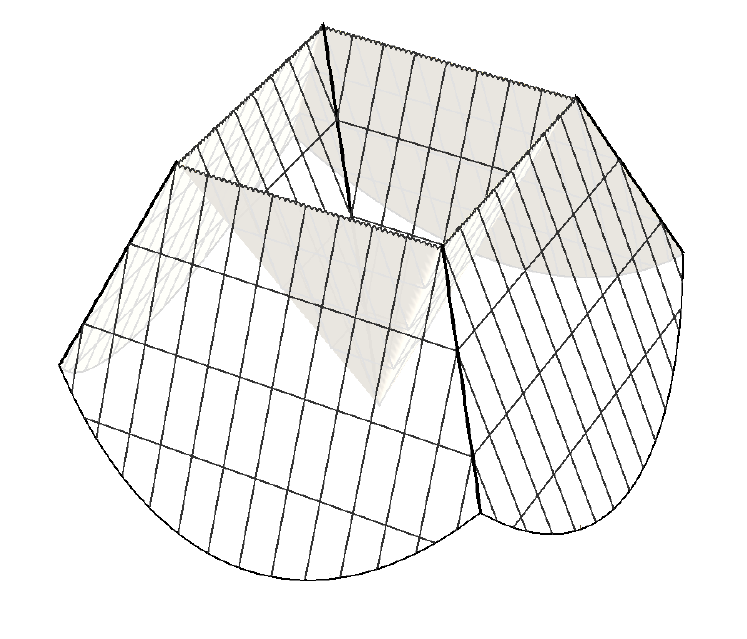}
 \caption{Sketch of the inverted pyramid used to achieve the $h^{5/3}$ upper bound. \label{fig:pyramid}}
\end{center}
\end{figure}

\begin{figure}[h]
\begin{center}
\includegraphics[width=.7\textwidth]{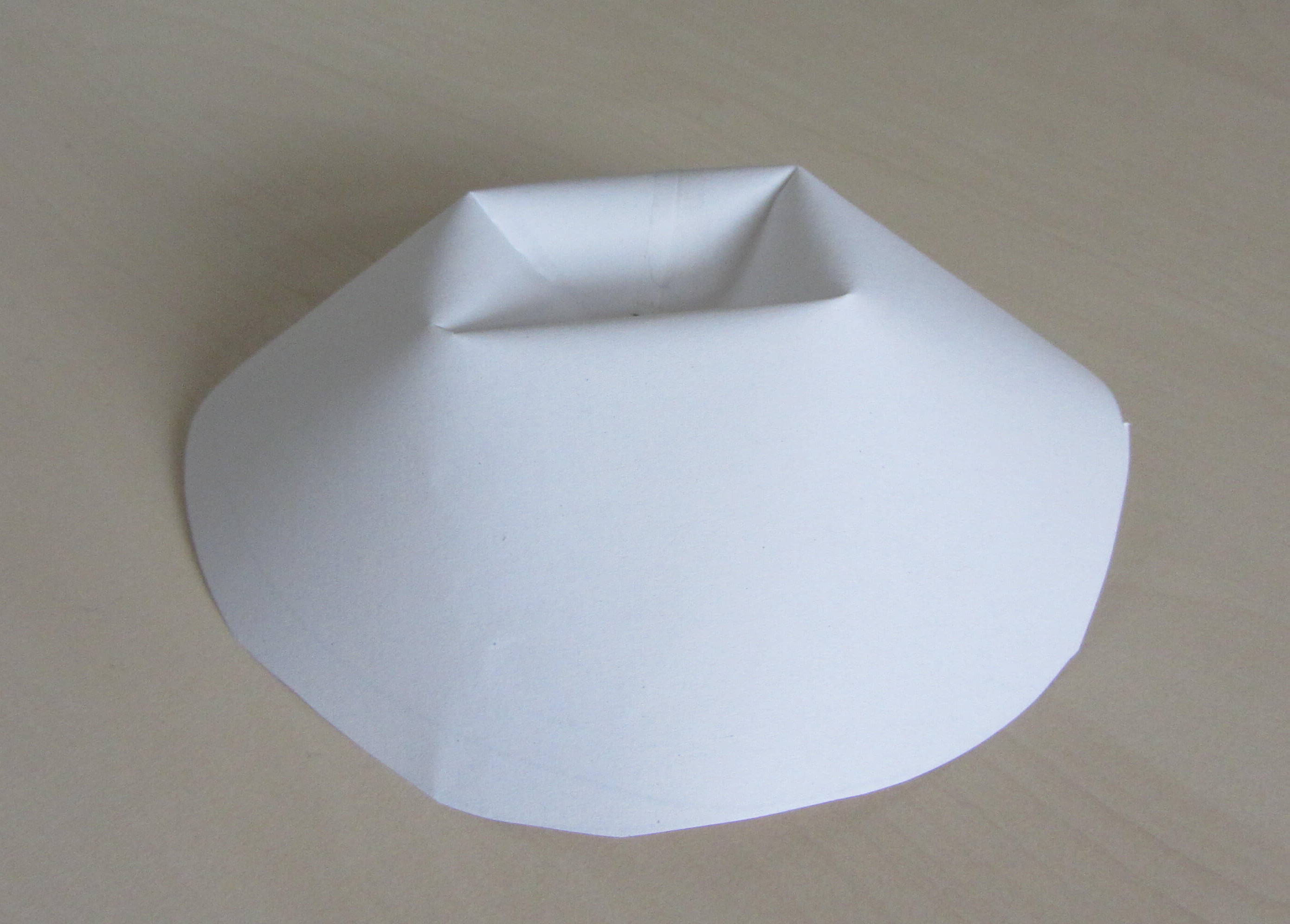}
\caption{An indented paper cone. \label{fig:indpapercone}}
\end{center}
\end{figure}

Precisely, the energy functional we study is given by 
$\tilde E_h:W^{1,2}(B_1;\R^2)\times W^{2,2}(B_1)\to \R$,
\begin{equation}
\tilde E_h(U,W)=\int_{B_1}\left|DU+DU^T +DW\otimes DW+\hat x^\bot\otimes
  \hat x^\bot\right|^2+h^2|D^2W|^2\d x\,,\label{eq:17}
\end{equation}
where  $\hat
x^\bot=(-x_2,x_1)/|x|$
and $B_1=\{x\in\R^2:|x|<1\}$. This expression  is obtained from the usual 
von K\'arm\'an  energy by rescaling and subtracting $\frac12\theta x^\perp$ from $U$, which corresponds within the linear
theory to the natural change of variables performed after removing the wedge from the plane, as discussed also at the beginning of Section \ref{secproof2}.
\\
Motivated by the results from \cite{PhysRevLett.93.255504}, we first
restrict ourselves to radially symmetric
configurations, and model the indentation by 
Dirichlet boundary conditions at the center and the
boundary. We consider the set of radially symmetric configurations given by
\begin{equation}\label{eq:changeofvar}
U(x)=\frac12\left(u(|x|)-|x|\right)\frac{x}{|x|}\hskip1cm \text{ and } \hskip1cm W(x)=w(|x|)\,,
\end{equation}
where 
the pair of scalar functions $(u,w)$ is taken from the set
\begin{equation}\label{eq:Adelta}
\A_\delta=\{(u,w) : (U,W)\in W^{1,2}(B_1;\R^2)\times W^{2,2}(B_1),
w(0)=0, w(1)=1-\delta\}.
\end{equation}
This definition makes sense: in particular, that $w$ has well-defined trace at $r=0,1$ follows since $w\in W^{1,2}((0,1))$. This is proved in Lemma \ref{lem:radsymmspace} in the appendix, which makes clear the exact regularity we are assuming on $u,w$.

Under these assumptions, and dropping an irrelevant factor of $2\pi$, the functional \eqref{eq:17} reduces to
$E_h:\A_\delta\to \R$,
\begin{equation}
E_h(u,w)=\int_0^1  \frac{u^2}{r}+r\left(u'+w'^2-1\right)^2+h^2
\left(rw''^2+\frac{w'^2}{r}\right)\d r\,.\label{eq:2}
\end{equation}
Our first result is the following. 
\begin{theorem}  
  \label{thm:main}
 There exists a numerical constant $C>0$ such that
\[
\begin{split}
  \frac{1}{C}
  \Big(h^2{\log\frac1h}+&\min\left\{\delta^2h^{1/2},\delta^{1/2}h^{3/2}\right\}\Big)\\
\leq&
  \min_{\A_\delta}E_h\leq {C}
  \left(h^2{\log\frac1h}+\min\left\{\delta^2h^{1/2},\delta^{1/2}h^{3/2}\right\}\right)
\end{split}
\]
for all $h\in (0, 1/2]$, $\delta\in [0,1]$.
\end{theorem}
\begin{proof}
 The upper bound follows from Proposition \ref{thm:UB}, which is proven 
 in  Section \ref{secUB}. 
 The lower bound follows from Proposition \ref{thm:LB}, which is proven 
 in  Section \ref{secLB}. 
\end{proof}
Our second result is that a lower elastic energy 
can be achieved if the requirement of radial symmetry is
dropped. 
Let the set of allowed configurations be
\[
\begin{split}
  \tilde\A_\delta=\{&(U,W)\in  W^{1,2}(B_1;\R^2)\times W^{2,2}(B_1):
    W(0)=0,|W|\leq 1-\delta\text{ on }\partial B_1\}\,.
\end{split}
\]
\begin{theorem}
\label{thm:ridgeconst}
There exists a numerical constant $C>0$ such that
\[
  \min_{\tilde \A_\delta}\tilde E_h\leq C h^{5/3}
\]
for all $h\in (0, 1/2]$, $\delta\in [0,1/2]$.
\end{theorem}
The proof is given in Section \ref{secproof2} below. One of the main ingredients
of the proof of Theorem \ref{thm:ridgeconst} is the construction of a smoothed
ridge that interpolates between flat faces, see Lemma \ref{lemma:ridge}. This lemma
is based on a  construction by the first author and Maggi for the geometrically
fully nonlinear case \cite{Maggi08},
and improves a similar lemma by Bedrossian and Kohn \cite{BedrossianKohn2015}. 
A previous construction by Venkataramani \cite{Venkataramani04} within a von K\'arm\'an  model, which contained 
several of the ideas then used in \cite{Maggi08}, cannot be used here since 
it had a different geometry, including in particular different reference configurations on the two sides of the fold.
In Remark \ref{rem:bedrkohn} below, we note that this improvement of the von
K\'arm\'an version of the smooth ridge construction also leads to an
improvement of Theorem 1.4 in \cite{BedrossianKohn2015}.

The following corollary proves that the minimizer within the class $\tilde A_\delta$ is not radially symmetric for small enough $h$. To this end, we define 
\[
\tilde A^\text{rad}_\delta = \tilde A_\delta \cap \left \{U(x) = \tilde U(\abs{x})\frac{x}{|x|},\ W(x) = \tilde W(\abs{x})\right\}.
\]

\begin{corollary} 
There is $c>0$ such that 
for all $\delta \in (0,1/2]$, we have that 
\[
 \min_{\tilde \A_\delta} \tilde E_h < \min_{\tilde \A^\text{rad}_\delta} \tilde E_h \quad \text{ whenever }0<h<c\delta^3.
\]
\end{corollary}
\begin{proof}
Changing variables as in \eqref{eq:changeofvar} and using  Theorem \ref{thm:main}  we find that 
\[
\min _{\tilde A^{\text{rad}}_\delta} \tilde E_h \gtrsim \inf_{\delta' \in [\delta, 1]} \min_{A_{\delta'}} E_h
\gtrsim
h^2{\log\frac1h}+\min\left\{\delta^2h^{1/2},\delta^{1/2}h^{3/2}\right\}
\]
Since by  Theorem \ref{thm:ridgeconst} we have 
$ \min_{\tilde \A_\delta} \tilde E_h\lesssim h^{5/3}$, the assertion follows.
\end{proof}

\subsection*{Notation}
The symbol $C$ denotes various numerical constants. The value of $C$ may change
from one line to the next, but never depends on the parameters $h$ and $\delta$. 
The notation ``$f\lesssim g$'' means that there exists  $C\in(0,\infty)$ such that $f\leq Cg$. \\
We denote the oscillation of a function $f$ on a set $A$ by
\[
\osc{A}\{f\}=\sup_{x,y\in A}\,\abs{f\left(y\right)-f\left(x\right)}.
\]

\section{Upper bounds for radially symmetric cones}\label{secUB}

\begin{figure}
 \begin{center}
  \includegraphics[width=12cm]{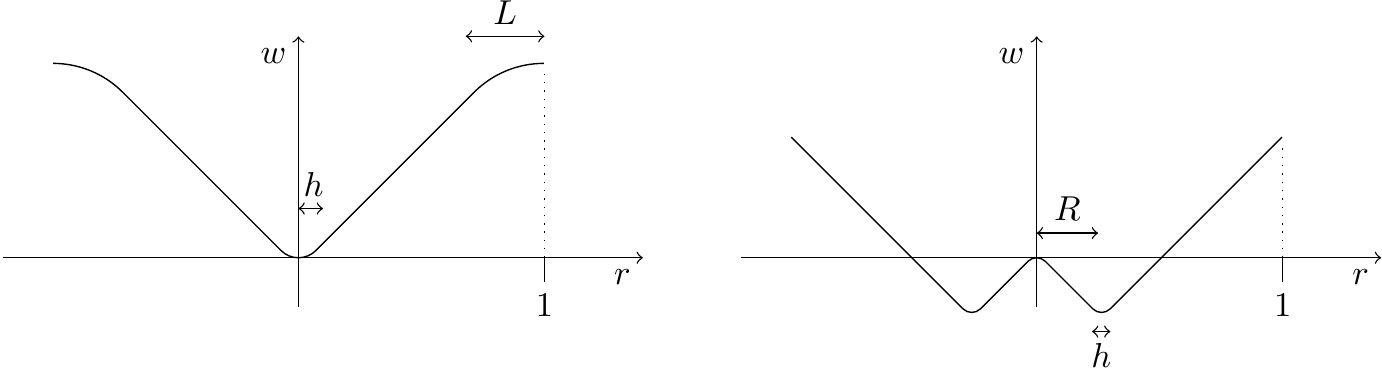}
 \end{center}
\caption{Sketch of the two upper bounds. Left panel: small-$\delta$ construction from Lemma  \ref{lem:UBii}. Right panel: large-$\delta$
construction from Lemma \ref{lem:UBi}.
For clarity we replicate the construction on the negative $r$-axis, so that the illustration can be understood
as the graph of the section $W(r,0)$ of the full three-dimensional picture.
}
\label{fig:constr}
\end{figure}
In this section we prove the upper bound contained in Theorem \ref{thm:main}. This is based on two different constructions,
which are relevant for small and large $\delta$, respectively, and which are illustrated in Figure \ref{fig:constr}.

For $\delta=0$ one starts from $w(r)=r$ and mollifies it close to the origin, on a length scale $h$. 
This gives rise to the well known $h^2\log \frac1h$ energy of the disclination, and already achieves a reduction
in total height
of order $h$. If $\delta\ge h$, then the larger indentation can be obtained by ``flattening'' the cone. The optimal way of doing this is by localizing the
modification to a neighborhood (of size $L$) of the free boundary, so that the $u'$ term can be used to compensate 
the error  in $(w')^2-1$. However, this generates an increasing error in $u$, which is penalized by the $u^2/r$
term in the energy. This mode of response is qualitatively linear, as no use is made of the nonconvexity of the energy,
and results in a total energy of order $\delta^2h^{1/2}$. This is illustrated in the left panel of Figure \ref{fig:constr}
and performed in Lemma \ref{lem:UBii} below.

For larger deformations it becomes convenient to use the two-well structure of the problem, setting $w'=-1$ in part of the domain.
In other words, one mixes the two constructions which would arise by $w(r)=-r$ and $w(r)=r$.
One then needs to construct an appropriate interface between $w'=-1$ and $w'=1$, and to use $u$ in an optimal way to minimize
the interfacial energy. This is illustrated in the right panel of Figure \ref{fig:constr}
and performed in Lemma \ref{lem:UBi} below.

Recall the definition of $A_\delta$ given in \eqref{eq:Adelta}.

\begin{proposition}
\label{thm:UB}There exists a numerical constant $C>0$ such that
\[
\min_{\A_{\delta}}\,E_{h}\leq 
C\left( h^{2}{\log\frac1h}+\min\left\{ \delta^{2}h^{1/2},\delta^{1/2}h^{3/2}\right\}  \right)
\]
whenever $0<h\leq \frac12$, $0\leq \delta\leq1$.
\end{proposition}
We prove this in two parts.
\begin{lemma}
\label{lem:UBi}We have that
\[
\min_{\A_{\delta}}\,E_{h}\lesssim h^{2} \log\frac1h+ \delta^{1/2}h^{3/2}
\]
whenever $0<h\leq\delta\leq1$ and $h\leq \frac12$.\end{lemma}
\begin{proof}
We claim that there exists $W_{0}\in C^{\infty}\left([-1,1]\right)$
with the following properties:
\[
W_{0}\left(\pm1\right)=0,\quad W_{0}'\left(\pm1\right)=\pm1,\quad\int_{-1}^{1}\abs{W_{0}'\left(t\right)}^{2}-1\,dt=0.
\]
Indeed, this can be achieved by mollifying the function $\overline{W}\left(x\right)=\abs{x}-1$
and adding a suitable multiple of a smooth, non-positive function
that is supported near the origin. Next, choose $\eta\in C^{\infty}\left([0,\infty)\right)$
such that 
\[
\eta\left(x\right)=0\ \text{for}\ x\leq\frac{1}{5},\quad\eta\left(x\right)=1\ \text{for}\ x\geq\frac{2}{5},\quad\fint_{0}^{2/5}\eta\left(x\right)\,dx=1.
\]
Let $R=\frac{1}{2}\delta$ and $l=\frac{1}{10}\sqrt{h\delta}$. Note
that for all $\delta\geq h$, 
\[
\frac{R-l}{h}=\frac{1}{2}\frac{\delta}{h}-\frac{1}{10}\sqrt{\frac{\delta}{h}}\geq\frac{2}{5}.
\]
Given these choices, we define
\[
w_{1}'\left(r\right)=\begin{cases}
-\eta\left(\frac{r}{h}\right) & 0\leq r\leq R-l\\
W_{0}'\left(\frac{r-R}{l}\right) & R-l\leq r\leq R+l\\
1 & R+l\leq r\leq1
\end{cases}
\]
and $w_{1}\left(r\right)=\int_{0}^{r}w_{1}'\left(t\right)\,dt$. Also,
we take 
\[
u_{1}\left(r\right)=\begin{cases}
0 & 0\leq r\leq R-l\\
\int_{R-l}^{r}1-\abs{w_{1}'\left(t\right)}^{2}\,dt & R-l\leq r\leq1
\end{cases}.
\]
Then, we have that $w_1(0)=0$ and that
\begin{align*}
w_{1}\left(1\right) & =\int_{0}^{R-l}-\eta\left(\frac{r}{h}\right)\,dr+\int_{R-l}^{R+l}W_{0}'\left(\frac{r-R}{l}\right)\,dr+\int_{R+l}^{1}1\,dr\\
 & =\int_{0}^{2h/5}-\eta\left(\frac{r}{h}\right)\,dr-\left(R-l-\frac{2h}{5}\right)+1-\left(R+l\right)\\
 & =1-2R=1-\delta.
\end{align*}
Thus, by Lemma \ref{lem:radsymmspace}, we have that $\left(u_{1},w_{1}\right)\in \A_{\delta}$.

Now we estimate the energy of this construction. For the membrane
terms, we have the estimates
\[
\int_{0}^{1}\frac{u_{1}^{2}}{r}\,dr=\int_{R-l}^{R+l}\frac{u_{1}^{2}}{r}\,dr\lesssim\frac{l^{3}}{R}\lesssim\delta^{1/2}h^{3/2}
\]
and
\[
\int_{0}^{1}\left(u'_1+w_{1}'^{2}-1\right)^{2}r\,dr=\int_{0}^{R-l}\left(\left|\eta\left(\frac{r}{h}\right)\right|^{2}-1\right)^{2}r\,dr\lesssim h^{2}.
\]
For the bending terms, we have the estimates 
\begin{align*}
\int_{0}^{1}w_{1}''^{2}r\,dr & =\int_{0}^{R-l}\left|\frac{1}{h}\eta'\left(\frac{r}{h}\right)\right|^{2}r\,dr+\int_{R-l}^{R+l}\left|\frac{1}{l}W_{0}''\left(\frac{r-R}{l}\right)\right|^{2}r\,dr\\
 & \lesssim1+\frac{R}{l}\lesssim\sqrt{\frac{\delta}{h}}
\end{align*}
and
\[
\int_{0}^{1}\frac{w_{1}'^{2}}{r}\,dr=\int_{\frac{h}{5}}^{1}\frac{w_{1}'^{2}}{r}\,dr\lesssim\log\frac{1}{h}.
\]
Therefore,
\[
E_{h}\left(u_{1},w_{1}\right)\lesssim \delta^{1/2}h^{3/2}+h^{2}\log\frac{1}{h}\,.
\]

\end{proof}

\begin{lemma}
\label{lem:UBii}We have that
\[
\min_{\A_{\delta}}\,E_{h}\lesssim \delta^{2}h^{1/2}+h^{2}{\log\frac1h}+\delta^{4}h^{-1/2}
\]
whenever $0\leq\delta\leq1$, $0<h\leq \frac12$.\end{lemma}
\begin{proof}
Let $\eta$ be as in the previous proof and let $L=h^{1/2}$. Note
that for all $h\leq \frac12$, 
\[
\frac{1-L}{h}=\frac{1-\sqrt{h}}{h}\geq\frac{2}{5}.
\]
Define 
\[
w_{2}'\left(r\right)=\begin{cases}
\eta\left(\frac{r}{h}\right) & 0\leq r\leq1-L\\
1-\frac{2\delta}{L^{2}}\left(r-\left(1-L\right)\right) & 1-L\leq r\leq1
\end{cases}
\]
and $w_{2}\left(r\right)=\int_{0}^{r}w_{2}'\left(t\right)\,dt$. Define
\[
u_{2}\left(r\right)=\begin{cases}
0 & 0\leq r\leq1-L\\
\int_{1-L}^{r}(1-\abs{w_{2}'\left(t\right)}^{2})dt & 1-L\leq r\leq1
\end{cases}.
\]
Then, we have that $w_2(0)=0$ and that
\begin{align*}
w_{2}\left(1\right) & =\int_{0}^{1-L}\eta\left(\frac{r}{h}\right)\,dr+\int_{1-L}^{1}1-\frac{2\delta}{L^{2}}\left(r-\left(1-L\right)\right)\,dr\\
 & =\int_{0}^{\frac{2h}{5}}\eta\left(\frac{r}{h}\right)\,dr+\left(1-L-\frac{2h}{5}\right)
 +L-\int_{0}^{L}\frac{2\delta}{L^{2}}r\,dr\\
 & =1-\delta.
\end{align*}
Thus, by Lemma \ref{lem:radsymmspace}, we have that $\left(u_{2},w_{2}\right)\in \A_{\delta}$.

Now we estimate the energy of this construction. First, we note that
\[
\abs{u_{2}\left(r\right)}\lesssim\delta\left(1+\frac{\delta}{L}\right)
\]
for all $r\in\left[1-L,L\right]$, so that the first membrane term
satisfies
\[
\int_{0}^{1}\frac{u_{2}^{2}}{r}\,dr=\int_{1-L}^{1}\frac{u_{2}^{2}}{r}\,dr\lesssim\delta^{2}L\left(1+\frac{\delta}{L}\right)^{2}\lesssim\delta^{2}h^{1/2}\left(1+\frac{\delta^{2}}{h}\right).
\]
The second membrane term satisfies
\[
\int_{0}^{1}\left(u'_{2}+w_{2}'^{2}-1\right)^{2}r\,dr=\int_{0}^{1-L}\left(\left|\eta\left(\frac{r}{h}\right)\right|^{2}-1\right)^{2}r\,dr\lesssim h^{2}.
\]
The bending terms satisfy 
\begin{align*}
\int_{0}^{1}w_{2}''^{2}r\,dr & =\int_{0}^{1-L}\left|\frac{1}{h}\eta'\left(\frac{r}{h}\right)\right|^{2}r\,dr+\int_{1-L}^{1}\left|\frac{2\delta}{L^{2}}\right|^{2}r\,dr \lesssim1+\frac{\delta^{2}}{L^{3}}=1 + \frac{\delta^{2}}{h^{3/2}}
\end{align*}
and
\[
\int_{0}^{1}\frac{w_{2}'^{2}}{r}\,dr=\int_{\frac{h}{5}}^{1}\frac{w_{2}'^{2}}{r}\,dr\lesssim\log\frac{1}{h}.
\]
Therefore,
\begin{align*}
E_{h}\left(u_{2},w_{2}\right) & \lesssim \delta^{2}h^{1/2}\left(1 + \frac{\delta^{2}}{h}\right)+h^{2}\left(1 + \frac{\delta^{2}}{h^{3/2}}\right)+h^{2}\log\frac{1}{h} \\
 & \lesssim \delta^{2}h^{1/2}+\frac{\delta^{4}}{h^{1/2}}+h^{2}\log\frac{1}{h}\, .
\end{align*}

\end{proof}

\begin{proof}[Proof of Proposition \ref{thm:UB}]
We only need to combine Lemma \ref{lem:UBi} and Lemma \ref{lem:UBii}. 
If $\delta\le h^{2/3}$ then Lemma \ref{lem:UBii} implies that
\[
\min_{\A_{\delta}}\,E_{h}\lesssim \delta^{2}h^{1/2}+h^{2}{\log\frac1h}+\delta^{4}h^{-1/2}
 \lesssim\delta^{2}h^{1/2}+h^{2}{\log\frac1h}
\]
since $\delta^4h^{-1/2}\le  h^{2}$ in this case, and the proof is concluded.

If instead $\delta > h^{2/3}$ then necessarily $h\le \delta$, so that Lemma \ref{lem:UBi} implies
the assertion.
\end{proof}

\section{Lower bounds for radially symmetric cones}\label{secLB}

In this section, we prove the lower bound contained in Theorem \ref{thm:main}:
\begin{proposition}
\label{thm:LB}There exists a numerical constant $C>0$ such that
\[
\frac{1}{C} \left(h^{2}{\log\frac1h}+\min\left\{ \delta^{2}h^{1/2},\delta^{1/2}h^{3/2}\right\}\right)  \leq\min_{\A_{\delta}}\,E_{h}
\]
whenever $0<h\leq \frac12$, $0\leq \delta\leq1$.
\end{proposition}
The key idea is to consider $\min_{\A_\delta}E_h$ as a two-well problem. Indeed, if the energy is small then $u$ has to be small,
at least in an $L^2$ sense. Therefore $u'$ is also small (in a weak sense) and $w'$ needs to be close to $+1$ or $-1$.
In turn, oscillations in $w'$ are penalized by the curvature terms.
Since  the length scale over which $w'$ varies 
is controlled, the $L^2$ control on $u$ gives an effective control on the derivative $u'$.
Therefore, at a certain level of abstraction we are dealing with a problem of the well-known
type of $\int_0^1 ((w')^2-1)^2 + h^2 (w'')^2 dx$.
The situation is  made more complex not only by the mentioned presence of $u$, but also by
the various factors $r$ and $1/r$, which make the behavior
close to the origin different. The region closest to the origin is treated separately (here the $(w')^2/r$ curvature
term dominates) and the rest will be decomposed into dyadic intervals, so that in each of them $r$ can be replaced by a constant.

In order to make the two-well structure quantitative, we define
 the ``wells''
\begin{equation}
W=W^{+}\cup W^{-},\quad W^{+}=\left(\frac{1}{2},\frac{3}{2}\right),\quad W^{-}=\left(-\frac{3}{2},-\frac{1}{2}\right),\label{eq:wells}
\end{equation}
and for any $f \in C((0,1])$ we define $\tau$  as the furthest point from the origin at which $f$ does
not belong to one of them,
\begin{equation}\label{eq:outofwell}
\tau_f = \max\left\{ t\in(0,1]:  f(t)\notin W\right\}.
\end{equation}
If $f \in W$ for all $t\in(0,1]$, then we set $\tau_f=0$. The maximum is otherwise 
attained, since $f$ is assumed to be continuous on $(0,1]$. We also remark that if $\tau_f\in (0,1)$ then $f(\tau_f)\in \partial W$.

The proof is separated into three parts, which correspond to the three expressions in the 
statement of Proposition \ref{thm:LB}. 

In the entire section we assume that
\begin{equation}\label{eqstandingass}
\text{ $\delta\in [0,1]$,\hskip3mm $h\in (0,1/2]$, \hskip3mm$(u,w)\in \A_\delta$, \hskip3mm$E_h=E_h(u,w)$}.
\end{equation}
In a slight abuse of notation, we will write $w'$ to mean both the Sobolev function and its continuous representative on $(0,1]$ (see Lemma \ref{lem:radsymmspace} in the appendix). Given this, we define
\begin{equation}\label{eq:furthestpt}
\tau = \tau_{w'},
\end{equation}
referring to the definition in (\ref{eq:outofwell}) above.

\subsection{Full conical inversion}

We begin the proof with a construction that will be used throughout 
to quantify the fact that ``low energy'' requires that $w'$ is mostly inside the wells.
Given $w\in W^{1,2}\left([0,1]\right)$ and $a\in\left[0,\frac{1}{2}\right]$,
we define 
\begin{equation}
I_{a}=\left[a,2a\right]\quad\text{and}\quad g_{a}(r)=\int_{a}^{r}(1-w'^{2})dt+c_{a}\label{eq:g_a}
\end{equation}
where $c_{a}$ is chosen such that $\fint_{I_{a}}g_{a}dt=0$. 
\begin{lemma}\label{lem:non-local-lb}
 Assume (\ref{eqstandingass}). For all $0\leq a\leq\frac{1}{2}$, we have
that
\[
\norm{g_{a}}_{L^{2}\left(I_{a}\right)}\lesssim a^{1/2}E_{h}^{1/2}\,.
\]
\end{lemma}
\begin{proof}
Let $v=u-g_{a}$. Then, using Poincare's inequality 
we have the chain of inequalities 
\begin{equation}
  \begin{split}
    E_{h} & \geq\int_{I_{a}}\frac{\left(v+g_{a}\right)^{2}}{r}dr+\int_{I_{a}}r\abs{v'}^{2}dr\\
    & \gtrsim\frac{1}{a}\int_{I_{a}}\abs{v+g_{a}}^{2}dr+a\int_{I_{a}}\abs{v'}^{2}dr\\
    & \gtrsim\frac{1}{a}\int_{I_{a}}\abs{v+g_{a}}^{2}dr+\left|v-\fint_{I_{a}}v\, dr'\right|^{2}dr\\
    & \gtrsim\frac{1}{a}\int_{I_{a}}\left|g_{a}+\fint_{I_{a}}v\,dr'\right|^{2}dr\,.
  \end{split}\label{eq:6}
\end{equation}
Furthermore, using that $\fint_{I_{a}}g_{a}dr=0$, we have that
\[
\int_{I_{a}}\left|g_{a}+\fint_{I_{a}}v\,dr'\right|^{2}dr = \int_{I_a} \abs{g_a}^2\,dr + a \left|\fint_{I_a} v dr\right|^2\\
\geq \int_{I_a} \abs{g_a}^2\,dr.
\]
Combining this with \eqref{eq:6} yields the claim of the lemma.
\end{proof}
Next, we quantify the fact that low energy implies that $w'$
is nearly constant. 

\begin{lemma}
\label{cor:invertedLB}
Assume (\ref{eqstandingass}), let $\tau$ be defined as in (\ref{eq:furthestpt}). Then
\[
\min\left\{ \tau^{1/2}h^{3/2},\tau^{2}\right\} \lesssim E_{h}\,.
\]
\end{lemma}
\begin{proof}
We can assume $\tau>0$. Then necessarily $w'(\tau)\not\in W$, 
so that either $|w'(\tau)|\le 1/2$ or $|w'(\tau)|\ge 3/2$.

For  $0<\epsilon\leq\frac{\tau}{2}$ chosen below
we define $I_{t,\epsilon}=\left[t-\epsilon,t\right]$.
Note that $I_{\tau,\epsilon}\subset\left[\frac{\tau}{2},\tau\right]=I_{\frac{\tau}{2}}$.
If 
\begin{equation}
\osc{I_{\tau,\epsilon}}\left\{ w'\right\} \leq\frac{1}{4},\label{eq:osc_bd}
\end{equation}
then one has either $|w'|\le 3/4$ or $|w'|\ge 5/4$ on $I_{\tau,\epsilon}$.
Therefore $\abs{g'_{\frac{\tau}{2}}}\gtrsim1$ on $I_{\tau,\epsilon}$, and by continuity
$g'_{\frac{\tau}{2}}$ does not change sign.
Then, by Lemma
\ref{lem:non-local-lb}, 
\[
E_{h}\gtrsim\frac{1}{\tau}\int_{I_{\frac{\tau}{2}}}g_{\frac{\tau}{2}}^{2}dt\geq\frac{1}{\tau}\int_{I_{\tau,\epsilon}}g_{\frac{\tau}{2}}^{2}dt\gtrsim\frac{\epsilon^{3}}{\tau}.
\]
Now we turn to the case that  (\ref{eq:osc_bd}) does not hold.
By Jensen's inequality and the fundamental theorem of calculus,
\[
\frac{E_{h}}{h^{2}}\geq\int_{I_{\tau,\epsilon}}\abs{w''}^{2}rdr\gtrsim \tau\int_{I_{\tau,\epsilon}}\abs{w''}^{2}\,dr\gtrsim\frac{\tau}{\epsilon}\left(\osc{I_{\tau,\epsilon}}\left\{ w'\right\}\right)^2 .
\]
Therefore in this case
\[
E_{h}\gtrsim\frac{\tau h^{2}}{\epsilon}.
\]
Combining the two cases, we find that
\[
E_{h}\gtrsim \min\left\{ \frac{\epsilon^{3}}{\tau},\frac{\tau h^{2}}{\epsilon}\right\} \hskip5mm
\text{ for any } \epsilon\in (0, \frac{\tau}{2}]\,.
\]
We finally select $\epsilon=\min\{\frac{1}{2}\tau,\sqrt{h\tau}\}$ to obtain the desired lower
bound
\begin{align*}
E_{h} & \gtrsim\min\left\{ \tau^{2},h^{3/2}\tau^{1/2},\frac{\tau h^{2}}{\sqrt{h\tau}}\right\} 
  =\min\left\{ \tau^{2},h^{3/2}\tau^{1/2}\right\} .
\end{align*}

\end{proof}

\subsection{Partially inverted lower bound}

In this section, we discuss the case where $w'$ stays in a well --
as given in (\ref{eq:wells}) -- throughout the bulk of the domain.
The main case is the one in which this well is $W_+$, the other one is
easily excluded. By bulk of the domain we mean the interval $[\frac\delta8,1]$.
One important ingredient is an estimate of $w'$ (or of $g_a$) in terms of the energy
via a Gagliardo-Nirenberg inequality in Lemma \ref{lem:well-cvg}. Since one variant 
of the inequality we
use is not completely standard, we provide a short proof in the Appendix.

We shall use the definitions of $I_{a}$ and $g_{a}$ given in (\ref{eq:g_a}),
and the definition of $W$ given in (\ref{eq:wells}).
\begin{lemma}
\label{lem:well-cvg}\label{lem:osc-g}
Assume (\ref{eqstandingass}), and assume that $a\in (0,1/2)$ is such that $w'(r)\in W$ for all
$r\in I_{a}$. Then we have the bounds 
\begin{align*}
\norm{1-w'^{2}}_{L^{2}\left(I_{a}\right)}\lesssim & 
\left(\frac1{a^{1/2}}+\frac1{h^{1/2}}\right)E_{h}^{1/2}\,,\\
\osc{I_{a}}\left\{ g_{a}\right\} \lesssim &\left(1 + \left(\frac{a}{h}\right)^{1/4}\right)E_{h}^{1/2}\,.
\end{align*}
\end{lemma}
\begin{proof}
To prove the first result, we apply the following standard Gagliardo-Nirenberg
interpolation inequality:
\[
\norm{g_{a}'}_{L^{2}\left(I_{a}\right)}\lesssim\norm{g_{a}}_{L^{2}\left(I_{a}\right)}^{1/2}\norm{g_{a}''}_{L^{2}\left(I_{a}\right)}^{1/2}+\frac{1}{\abs{I_{a}}}\norm{g_{a}}_{L^{2}\left(I_{a}\right)}.
\]
By the definition of $g_a$ and $W$, 
\begin{equation}
\norm{g_{a}''}_{L^{2}\left(I_{a}\right)}\lesssim\norm{w'w''}_{L^{2}\left(I_{a}\right)}\lesssim\norm{w''}_{L^{2}\left(I_{a}\right)}
\lesssim \left( \frac{E_h}{a h^2}\right)^{1/2}\,.
\label{eq:g_a''}
\end{equation}
Therefore, by Lemma \ref{lem:non-local-lb} 
we have that 
\[
\norm{g_{a}'}_{L^{2}\left(I_{a}\right)}\lesssim \left(\frac1{a^{1/2}}+\frac1{h^{1/2}}\right)E_{h}^{1/2}.
\]
The desired inequality follows from the definition of $g_{a}$.

To prove the second result, we apply the interpolation inequality
given in Lemma \ref{lem:interp_osc} in the appendix  to $g_{a}$:
\[
\osc{I_{a}}\left\{ g_{a}\right\} \lesssim\norm{g_{a}}_{L^{2}\left(I_{a}\right)}^{3/4}\norm{g_{a}''}_{L^{2}\left(I_{a}\right)}^{1/4}+\frac{1}{a^{1/2}}\norm{g_{a}}_{L^{2}\left(I_{a}\right)}.
\]
Recalling (\ref{eq:g_a''}) and Lemma \ref{lem:non-local-lb} concludes the proof.
\end{proof}
Since we only know that $w'$ resides in the wells in the bulk, we
will need a separate argument to control it near the origin. In the
next result we make some constants explicit, since they will be used
in the subsequent. 
\begin{lemma}
\label{lem:ctrl-origin} 
Assume (\ref{eqstandingass}), and let $0<s\leq t\leq\frac{1}{2}$. Then we
have the bounds 
\begin{align*}
\norm{w'}_{L^{1}\left([0,s]\right)} & \leq\frac{s}{h}E_{h}^{1/2}\\
\norm{w'}_{L^{1}\left([s,t]\right)} & \leq Ct^{1/2}\left(1 + {\log\frac1s}\right)^{1/4}E_{h}^{1/4}+2t.
\end{align*}
\end{lemma}
\begin{proof}
Using the Cauchy-Schwarz inequality,
\[
\int_{0}^{s}\abs{w'}\,dr\leq\left(\int_{0}^{s}\frac{w'^{2}}{r}\,dr\right)^{1/2}\left(\int_{0}^{s}r\,dr\right)^{1/2}\leq\frac{E_{h}^{1/2}}{h}s\,.
\]
This proves the first inequality.

The proof of the second inequality is more involved. We begin by controlling
the in-plane displacement, $u$, at conveniently chosen points. For all
$\epsilon>0$ we have the elementary estimate
\[
\L^1\left(\left\{ r\in\left[0,1\right]\ :\ \abs{u(r)}>\sqrt{\epsilon r}\right\} \right)<\frac{1}{\epsilon}\int_{0}^{1}\frac{u^{2}}{r}\,dr\leq\frac{1}{\epsilon}E_{h}\,,
\]
where $\L^1$ denotes the Lebesgue measure on $\R$.
Calling $\gamma=\epsilon^{-1}E_{h}$, we can rearrange to find that
\[
\L^1\left(\left\{ r\in\left[0,1\right]\ :\ \abs{u(r)}>\left(\frac{r}{\gamma}\right)^{1/2}E_{h}^{1/2}\right\} \right)<\gamma
\]
for all $\gamma>0$. Choosing $\gamma=\frac{s}{2}$ and $\gamma=t$,
we find that there exist $r_{0}\in\left[\frac{s}{2},s\right]$ and
$r_{1}\in\left[t,2t\right]$ such that 
\[
\abs{u\left(r_{0}\right)} + \abs{u\left(r_{1}\right)}\leq3E_{h}^{1/2}.
\]

Now we estimate $w'$. Using Cauchy-Schwarz and the definitions of
$r_{0},r_{1}$,
\[
\int_{s}^{t}\abs{w'}\,dr\leq\int_{r_{0}}^{r_{1}}\abs{w'}\,dr\leq\left(\int_{r_{0}}^{r_{1}}w'^{2}\,dr\right)^{1/2}\sqrt{2t}.
\]
By the triangle inequality and another application of Cauchy-Schwarz,
\begin{align*}
\int_{r_{0}}^{r_{1}}w'^{2}\,dr & \leq\left|\int_{r_{0}}^{r_{1}}w'^{2}-1+u'\,dr\right|+\left|r_{1}-r_{0}\right|+\abs{u\left(r_{1}\right)-u\left(r_{0}\right)}\\
 & \leq\left(\int_{r_{0}}^{r_{1}}\abs{w'^{2}-1+u'}^{2}r\,dr\right)^{1/2}\left(\log\frac{r_{1}}{r_{0}}\right)^{1/2}+\abs{r_{1}-r_{0}}+\abs{u\left(r_{1}\right)-u\left(r_{0}\right)}.
\end{align*}
By our choice of $r_{0},r_{1}$, and since $s\leq\frac{1}{2}$, we
conclude that
\[
\int_{r_{0}}^{r_{1}}w'^{2}\,dr\leq C E_{h}^{1/2}\left(1+\left(\log\frac1s\right)^{1/2}\right)+2t\,.
\]
Assembling the estimates, we find that
\begin{align*}
\int_{s}^{t}\abs{w'}\,dr & \leq\left(C E_{h}^{1/2}\left(1+{\log\frac1s}\right)^{1/2}+2t\right)^{1/2}\sqrt{2t}\\
 & \leq C E_{h}^{1/4}t^{1/2}\left(1 + {\log\frac1s}\right)^{1/4}+2t.
\end{align*}
This completes the proof.
\end{proof}
With these estimates in hand, we can prove a lower bound on the energy.
We start with the case where $w'\in W_{+}$ in the bulk. 
\begin{lemma}
\label{cor:goodwell_LB}
Assume (\ref{eqstandingass}), and additionally $h\leq\delta$ and  $w'\left(r\right)\in W_{+}$
for all $r\in\left[\frac{\delta}{8},1\right]$. Then,
\[
\delta\lesssim\frac{1}{h^{1/4}}E_{h}^{1/2}+\delta^{1/2}\left({\log\frac1h}\right)^{1/4}E_{h}^{1/4}+\frac{1 + {\log\frac1\delta}}{h}E_{h}\,.
\]
\end{lemma}
\begin{proof}
First, we observe that the elementary inequality
\[
X\leq\frac{1}{2}\left(\left(2X-X^{2}\right)+X^{2}\left(2-X\right)^{2}\right)
\]
holds for all $X\leq1$. Therefore, if $w'\in W_{+}$ then we may
choose $X=1-w'$ to conclude the bound
\[
1-w'\leq\frac{1}{2}\left(1-w'^{2}+\left(1-w'^{2}\right)^{2}\right).
\]
Now let $N\in\mathbb{N}$ be such that $a=\frac{1}{2^{N}}$ satisfies
$2^{2}a<\delta\leq2^{3}a$. Then, we have that
\begin{align*}
\delta & =\int_{0}^{1}1-w'\,dt\\
 & =\int_{0}^{a}1-w'\,dt+\int_{a}^{1}1-w'\,dt\\
 & \leq\int_{0}^{a}1-w'\,dt+\int_{a}^{1}\frac{1}{2}\left(1-w'^{2}+\left(1-w'^{2}\right)^{2}\right)\,dt\\
 & \leq\frac{\delta}{4}+\int_{0}^{a}\abs{w'}\,dt+\frac{1}{2}\left|\int_{a}^{1}1-w'^{2}\,dt\right|+\frac{1}{2}\int_{a}^{1}\abs{1-w'^{2}}^{2}\,dt\\
 & =\frac{\delta}{4}+\left(i\right)+\abs{\left(ii\right)}+\left(iii\right).
\end{align*}
Consider the term $\left(i\right)$. Applying Lemma \ref{lem:ctrl-origin}
with the choices $s=\frac{h}{4}$, $t=\frac{\delta}{4}$, we conclude
that 
\begin{align*}
\left(i\right) & =\int_{0}^{a}\abs{w'}\,dt\leq\int_{0}^{\frac{h}{4}}\abs{w'}\,dt+\int_{\frac{h}{4}}^{\frac{\delta}{4}}\abs{w'}\,dt\\
 & \leq\frac{1}{4}E_{h}^{1/2}+C\delta^{1/2}\left(1 + {\log\frac4h}\right)^{1/4}E_{h}^{1/4}+\frac{\delta}{2}.
\end{align*}
Therefore, after a rearrangement, we find that
\begin{equation}
\delta\lesssim E_{h}^{1/2}+\delta^{1/2}\left({\log\frac1h}\right)^{1/4}E_{h}^{1/4}+\abs{\left(ii\right)}+\left(iii\right).\label{eq:estondelta}
\end{equation}

Now we bound $\left(ii\right)$ and $\left(iii\right)$. We have that
\[
\left(ii\right)=\frac{1}{2}\int_{a}^{1}(1-w'^{2})dt=\frac{1}{2}\sum_{j=1}^{N}\int_{2^{-j}}^{2^{-\left(j-1\right)}}(1-w'^{2})dt=\frac{1}{2}\sum_{j=1}^{N}g_{a_{j}}\left(a_{j-1}\right)-g_{a_{j}}\left(a_{j}\right)
\]
where $a_{j}=2^{-j}$. By Lemma \ref{lem:osc-g} and since $a_{j}\gtrsim\delta\geq h$,
\[
\abs{g_{a_{j}}\left(a_{j-1}\right)-g_{a_{j}}\left(a_{j}\right)}\lesssim\left(\frac{a_{j}}{h}\right)^{1/4}E_{h}^{1/2}.
\]
Therefore, by the triangle inequality 
\[
\abs{\left(ii\right)}\lesssim\sum_{j=1}^{N}\left(\frac{a_{j}}{h}\right)^{1/4}E_{h}^{1/2}\lesssim\frac{1}{h^{1/4}}E_{h}^{1/2}.
\]
In a similar way, we can use Lemma \ref{lem:well-cvg} to conclude
that
\[
\left(iii\right)=\frac{1}{2}\sum_{j=1}^{N}\int_{I_{a_{j}}}\abs{1-w'^{2}}^2dt\lesssim\frac{N}{h}E_{h}\lesssim\frac{1 + 
{\log\frac1\delta}}{h}E_{h}.
\]
Combining the estimates on $\left(ii\right),\left(iii\right)$ with
(\ref{eq:estondelta}) above gives that
\[
\delta\lesssim\delta^{1/2}\left({\log\frac1h}\right)^{1/4}E_{h}^{1/4}+\frac{1}{h^{1/4}}E_{h}^{1/2}+\frac{1 + {\log\frac1\delta}}{h}E_{h}.
\]

\end{proof}
Now we deal with the case where $w'\in W_{-}$ in the bulk. As one
might expect, this case leads to relatively large energy.
\begin{lemma}
\label{cor:badwell_LB}
Assume (\ref{eqstandingass}), and additionally $h\leq\delta$ and  $w'\left(r\right)\in W_{-}$
for all $r\in\left[\frac{\delta}{8},1\right]$. Then,
\[
\min\left\{ \frac{1}{\delta^{2}{\log\frac1h}},1\right\} \lesssim E_{h}.
\]
\end{lemma}
\begin{proof}
Using the definition of $W_{-}$ and the boundary conditions on $w$,
we have that
\[
1-\delta=\int_{0}^{1}w'dt=\int_{0}^{\frac{\delta}{8}}w'dt+\int_{\frac{\delta}{8}}^{1}w'dt\leq\int_{0}^{\frac{\delta}{8}}\abs{w'}\,dt-\frac{1}{2}\left(1-\frac{\delta}{8}\right).
\]
Then, since $\delta\leq1$, 
\[
\frac{7}{16}\leq\int_{0}^{\frac{\delta}{8}}\abs{w'}\,dt.
\]
Now we argue as in the proof above: using Lemma \ref{lem:ctrl-origin}
with the choices $s=\frac{h}{8}$, $t=\frac{\delta}{8}$, and again
using that $\delta\leq1$, it follows that 
\begin{align*}
\int_{0}^{\frac{\delta}{8}}\abs{w'}\,dt & =\int_{0}^{\frac{h}{8}}\abs{w'}\,dt+\int_{\frac{h}{8}}^{\frac{\delta}{8}}\abs{w'}\,dt\\
 & \leq\frac{1}{8}E_{h}^{1/2}+C\delta^{1/2}\left(1 + {\log\left(\frac8{h}\right)}\right)^{1/4}E_{h}^{1/4}+\frac{1}{4}.
\end{align*}
Hence, 
\[
1\lesssim E_{h}^{1/2}+\delta^{1/2}\left({\log\frac1h}\right)^{1/4}E_{h}^{1/4}
\]
and this implies the result.
\end{proof}

\subsection{Conical lower bound}

We require an additional lower bound to complete the proof of Theorem
\ref{thm:LB}. The bound given here is the well-known scaling law
for an unconstrained regular cone. Taking advantage of radial symmetry, we give
a quick proof using only the techniques developed above.
\begin{lemma}
\label{lem:conicalLB}
Assume (\ref{eqstandingass}). Then
\[
h^{2}\log\frac{1}{h}\lesssim E_{h}\,.
\]
\end{lemma}
\begin{remark}
We note here that this result can also be proved as a special case of the second
author's recent work on the scaling law of a regular cone \cite{olbermann2015energy}
which does not assume radial symmetry.\end{remark}
\begin{proof}
Let $\tau$ be as defined in (\ref{eq:furthestpt}). Let $\tau_{h}=h\left({\log\frac1h}\right)^{2}$
and consider the cases $\tau<\tau_{h}$ and $\tau\geq\tau_{h}$. In
the first case, we observe that $\abs{w'}\gtrsim1$ on $\left[\tau_{h},1\right]$
since $w'\in W$ there. Therefore, 
\[
\frac{E_{h}}{h^{2}}\geq\int_{\tau_{h}}^{1}\abs{w'}^{2}\frac{1}{r}\,dr\gtrsim\log\frac{1}{\tau_{h}}.
\]
Since
\[
\log\frac{1}{\tau_{h}}=\log\frac{1}{h}-2\log\log\frac{1}{h}\gtrsim\log\frac{1}{h},
\]
we conclude the lower bound. In the second case, $\tau\geq\tau_{h}$,
and we can apply  Lemma \ref{cor:invertedLB} to prove that
\[
E_{h}\gtrsim\min\left\{ \tau_{h}^{1/2}h^{3/2},\tau_{h}^{2}\right\} \gtrsim h^{2}{\log\frac1h}.
\]

\end{proof}

\subsection{Proof of Proposition \ref{thm:LB}}

Now we assemble the previous estimates to complete the proof of Proposition
\ref{thm:LB}.
\begin{proof}[Proof of Proposition \ref{thm:LB}]
By Lemma \ref{lem:conicalLB}
we have
\[
h^{2}\log\frac{1}{h}\lesssim E_{h}\,.
\]
If $\delta< h$ the proof is concluded, since both $ \delta^{2}h^{1/2}$ and $\delta^{1/2}h^{3/2}$ are in this case smaller than $h^2$. 
Therefore it suffices to prove that 
\begin{equation}\label{eqclaimprop2}
\min\left\{ \delta^{2}h^{1/2},\delta^{1/2}h^{3/2}\right\} \lesssim E_{h} \hskip1cm
\text{ whenever $\delta\geq h$.} 
\end{equation}
Let $\tau$ be as defined in (\ref{eq:furthestpt}).
Consider the cases $\tau\geq\frac{\delta}{8}$ and $\tau\leq\frac{\delta}{8}$.
In the first case, we apply Lemma \ref{cor:invertedLB} to find
that
\[
E_{h}\gtrsim\min\left\{ \delta^{1/2}h^{3/2},\delta^{2}\right\} =\delta^{1/2}h^{3/2}.
\]
In the second case, we have that $w'\left(r\right)\in W$ for all
$r\in\left[\frac{\delta}{8},1\right]$. Since $w'$ is continuous (see Lemma \ref{lem:radsymmspace} in the appendix),  either  $w'\left(r\right)\in W^{+}$
for all $r\in\left[\frac{\delta}{8},1\right]$, or  $w'\left(r\right)\in W^{-}$
for all $r\in\left[\frac{\delta}{8},1\right]$. 

Suppose that $w'\in W^{+}$ on $\left[\frac{\delta}{8},1\right]$.
Then by Lemma \ref{cor:goodwell_LB},
\[
\delta\lesssim\frac{1}{h^{1/4}}E_{h}^{1/2}+\delta^{1/2}\left( {\log\frac1h}\right)^{1/4}E_{h}^{1/4}+\frac{1 + {\log\frac1\delta}}{h}E_{h}.
\]
It immediately follows that
\begin{align*}
E_{h} & \gtrsim\min\left\{ \delta^{2}h^{1/2},\frac{\delta^{2}}{
{\log\frac1h}},\frac{\delta h}{1 + {\log\frac1\delta}}\right\} \\
 & \gtrsim\min\left\{ \delta^{2}h^{1/2},\frac{\delta h}{1 + {\log\frac1\delta}}\right\} .
\end{align*}
On the other hand, if $w'\in W^{-}$ on $\left[\frac{\delta}{8},1\right]$,
then we may apply Lemma \ref{cor:badwell_LB} to find that
\[
E_{h}\gtrsim\min\left\{ \frac{1}{\delta^{2}{\log\frac1h}},1\right\} .
\]

Combining all the cases, we find that 
\begin{align*}
E_{h} & \gtrsim\min\left\{ \delta^{1/2}h^{3/2},\delta^{2}h^{1/2},\frac{\delta h}{1 + {\log\frac1\delta}},\frac{1}{\delta^{2}{\log\frac1h}},1\right\}\,.
\end{align*}
Since $\min\{\delta^{1/2}h^{3/2},\delta^{2}h^{1/2}\}\lesssim h \delta^{5/4} \lesssim h \delta/(1+\log\frac1\delta)$, we have proved \eqref{eqclaimprop2}. 
\end{proof}

\section{Proof of Theorem \ref{thm:ridgeconst}}
\label{secproof2}
The main idea in the proof of Theorem \ref{thm:ridgeconst} is to construct a piecewise affine map with the appropriate 
large-scale behavior, and then to replace the sharp edges with smooth ridges. The same procedure was used 
for a single fold in \cite{Venkataramani04},
for generic compressive deformations in nonlinear elasticity in \cite{Maggi08}, and adapted to the von K\'arm\'an
setting in \cite{BedrossianKohn2015}, albeit with a restriction on the size of
the angle of the fold. In Lemma \ref{lemma:ridge} below, we remove this
restriction. 

For this construction it is useful to use a  reference configuration without intrinsic curvature. 
In the nonlinear setting this should be imagined as the flat graphene sheet with
a wedge removed. In the linear theory, the reference configuration still is the whole plane, but the in-plane displacement has a jump. Precisely, we define
$V(x)=U(x)+ \frac12\theta x^\bot$, where $\theta\in(-\pi,\pi)$ is the angle such that $ x=|x|(\cos\theta,\sin\theta)$. This introduces a singularity on
the ray $\Sigma=(-\infty,0]e_1$.
The corresponding change of variables in the energy leads to the functional
\begin{equation}
  \bar E_h[V,W,\Omega]=\int_{\Omega}\left|DV+DV^T +DW\otimes DW\right|^2+h^2|D^2W|^2\d x\,.
\end{equation}
If $V$ and $W$ obey the boundary conditions $(V^+-V^-)(x_1,0)=(0,\pi x_1)$, $W^+(x_1,0)=W^-(x_1,0)$, and $DW^+(x_1,0)=DW^-(x_1,0)$ for all $x_1<0$,
then one can reverse the transformation and obtain admissible functions for $\tilde E_h$,
satisfying
$\tilde E_h(U,W)=\bar E_h[V,W,B_1\setminus\Sigma]$.

\begin{figure}
 \centerline{\includegraphics[width=8cm]{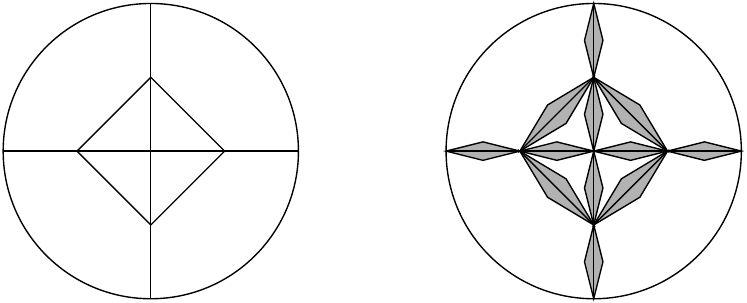}}
 \caption{Left: sketch of the regions where $V$ and $W$ are affine. Right: the twelve quadrilaterals where they are modified to obtain
 $V_h$, $W_h$.}
 \label{figmesh}
\end{figure}

\begin{proof}[Proof of Theorem \ref{thm:ridgeconst}]
By the definition of $\tilde\A_\delta$
it suffices to consider $\delta=1/2$. Further, we can assume that  $h$ is  small, otherwise one can set $U=0$ and $W=0$.
We define, for $\alpha>0$ chosen below,
 \begin{equation}
  W(x)=\alpha \min\{|x_1|+|x_2|, 1-|x_1|-|x_2|\}
 \end{equation}
and
\begin{equation}
 V(x)=\begin{cases}
           \displaystyle -\frac{\alpha^2}2  (x_1+x_2,x_1+x_2) & \text{ if } x_1\geq 0 \text{ and } x_2\geq 0\\[2mm]
           \displaystyle -\frac{\alpha^2}2  (x_1+x_2,-3x_1+x_2) & \text{ if } x_1<0 \text{ and } x_2\geq 0\\[2mm]
           \displaystyle -\frac{\alpha^2}2  (x_1-3x_2,5x_1+x_2) & \text{ if } x_1<0 \text{ and } x_2<0\\[2mm]
           \displaystyle -\frac{\alpha^2}2  (x_1-3x_2,x_1+x_2) & \text{ if } x_1\geq 0 \text{ and } x_2<0\,.
           \end{cases}
\end{equation}
One can check that $DV+DV^T+DW\otimes DW=0$, and that $V$ is continuous for $x_1=0$,
as well as for $x_2=0$, $x_1>0$. For $x_2=0$ and $x_1<0$ one obtains a jump of $4\alpha^2 (0,x_1)$. Therefore,
choosing $\alpha=\sqrt{\pi/4}$ ensures that $V$ is continuous on $\R^2$.
 
This configuration has gradients which jump across 12 segments. Around each segment we shall use a smooth ridge configuration as in Lemma \ref{lemma:ridge}.
We select 12 convex quadrilaterals with the mentioned segments as diagonals. We define $V_h$, $W_h$ as $V, W$ on the rest of the unit ball, 
and as the result of the Lemma in the 10 quadrilaterals which do not intersect the jump set of $V$.
For the remaining two we consider instead the function $V-4\alpha^2x^\perp\chi_{x_2>0}$, construct $V_h$ and $W_h$, and then add $4\alpha^2x^\perp\chi_{x_2>0}$ 
to the result.
At this point we define $U_h(x)=V_h(x)-\frac12 \theta x^\bot$. 

Now let $z_1,\dots,z_9\in \overline{B_1}$ denote the endpoints of the 12 segments.
The energy satisfies
\begin{equation}
 \bar E_h[V_h, W_h, B_1\setminus(\Sigma \cup \bigcup_{i=1}^9 B_h(z_i))] \le C h^{5/3}.
\end{equation}
It remains  to smoothen our construction on $\bigcup_{i=1}^9 B_h(z_i)$. 
For notational simplicity we discuss only the treatment of $z_1=(0,0)$.
Let $\psi_h\in C^\infty_c(B_{h})$ be such that $\psi=1$ on $B_{h/2}$. Let
$\varphi\in C^\infty_c(B_1)$ be a standard mollifier, and let $\varphi_h(\cdot)=h^{-2}\varphi(\cdot/h)$.
We set $\tilde W_h=\psi_h (W_h\ast \varphi_h) + (1-\psi_h)W_h$.
Then $\|D^2\tilde W_h\|^2_{L^2(B_{h})} \le C$. At the same time, $DW_h$ is uniformly bounded, therefore 
$\|DW_h\otimes DW_h-D\tilde W_h\otimes D\tilde W_h\|^2_{L^2(B_{h})} \le C h^2$. We conclude that
\begin{equation}
\bar  E_h[V_h, \tilde W_h, B_1\setminus \Sigma] \le C h^{5/3}.
\end{equation} 
This proves the theorem.
\end{proof}
Given two or more points $a_1, \dots, a_n$ we denote by $[a_1 \dots a_n]$ their convex envelope.

\begin{lemma}\label{lemma:ridge}
 Let $[abcd]\subset \R^2$ be a nondegenerate quadrilateral with diagonals $[ac]$
 and $[bd]$, contained in the square with diagonal $[ac]$. Furthermore, let
 $(V,W):[abcd]\to \R^2\times \R$ be continuous functions,
 affine on $[abc]$ and $[adc]$, with $DV+DV^T+DW\otimes DW=0$. Then for all $h\in (0,l/8)$ there are $V^h, W^h\in C^2([abcd]\setminus\{a,c\};\R^2\times \R)$ 
 such that $V=V^h$, $W=W^h$, $DW=DW^h$ on $\partial[abcd]\setminus\{a,c\}$, $|DV^h|+|DW^h|\le C$ and
 \begin{equation}
  \bar E_h[V^h, W^h, [abcd]\setminus (B_h(a) \cup B_h(c))]\le C h^{5/3}l^{1/3}\,,
 \end{equation}
where $l=|a-c|$, and
the constant $C$ may depend on the angles of $[abcd]$, and on $V$ and $W$, but not on $h$.
\end{lemma}
\begin{remark}
\label{rem:bedrkohn}
  The lemma is an improvement of a construction by Bedrossian and Kohn
  \cite{BedrossianKohn2015}, which only works for small $|DW|$ (small
  fold angle). In both our work and \cite{BedrossianKohn2015}, the proof consists of
  an adaptation of the geometrically fully nonlinear construction by the first
  author and Maggi \cite{Maggi08}.
We note that using our Lemma \ref{lemma:ridge}, Theorem 1.4 in 
\cite{BedrossianKohn2015} can be improved: the restriction on the misfit between
substrate and membrane -- assumption (1.17) in \cite{BedrossianKohn2015} -- can be omitted.
\end{remark}

Before we come to the proof, we  recall very briefly the general idea for the
geometrically fully nonlinear construction from
\cite{Maggi08}. Our proof is simply an adaptation of theirs to the von
K\'arm\'an setting.

Let us assume that the gradient discontinuity of the sharp fold is located on the set $\{(x,0):0\leq
x\leq l\}$, and that the sharp fold is given by 
\[
u^*(x_1,x_2)=\begin{cases}(x_1,x_2\cos\varphi_1,x_2\sin\varphi_1)&\text{ if }    x_2\geq 0\\
(x_1,x_2\cos\varphi_2,x_2\sin\varphi_2)&\text{ if }
    x_2< 0\,,\end{cases}
\]
with $\varphi_1\neq \varphi_2$. Note that $u^*$ is a Lipschitz isometry.
The first idea is to note that
the sharp fold $u^*$ can be approximated by a one-dimensional
construction. We have 
\begin{equation}
u^*(x_1,x_2)=x_1e_1+\gamma(x_2)\,,\label{eq:7}
\end{equation}
where $\gamma$ is a piecewise
linear curve. Approximations of $\gamma$ can be obtained by a suitable
smoothing, and these approximations $\tilde \gamma$ will result in a deformation of finite
energy when reinserted into \eqref{eq:7}. 
The second idea is to vary the width of the smoothed fold along its
length $l$: We  modify the sharp fold only in
the region $-f(x_1)\leq x_2\leq f(x_1)$, where the  function
$f:[0,l]\to[0,\infty)$ is chosen appropriately. This means, of course, that we
have to vary the curve $\tilde \gamma$ with $x_1$. This is achieved by a
suitable $x_1$-dependent rescaling of the argument of $\tilde \gamma$. Finally,
one has to  allow  for a correction $\beta(x_1,x_2)$ of the first component
of $u^*$, that vanishes for $|x_2|\geq f(x_1)$. All these ingredients are also
found in our ansatz for the von K\'arm\'an setting, equation \eqref{eq:8}
below. Note that due to the lack of full invariance under $SO(3)$ in the von
K\'arm\'an case, we cannot
assume that the third component $W_h$ of the deformation vanishes along the
discontinuity set of the gradient of the sharp fold, as we did for the fully
nonlinear case above. This somewhat obfuscates
the underlying idea in \eqref{eq:8}, and the reader is advised to check for her-
or himself that
\eqref{eq:8} is indeed nothing but the partially linearized version of the
ansatz we have just described.

\begin{proof}[Proof of Lemma \ref{lemma:ridge}]
We introduce the notation $[abcd]_h=[abcd]\setminus (B_h(a) \cup B_h(c))$.
By suitable translations and rotations of the domain (i.e., by making the replacements $V(x)\to
R^TV(Rx+z)$ and $W(x)\to W(Rx+z)$ for suitable $R\in SO(2)$ and $z\in \R^2$),  we may assume $a=0$ and
$c=(l,0)$. By adding suitable constants, we may  assume
$V(0)=W(0)=0$. Finally, by making the replacement $V(x)\to V(x)+\lambda x^\bot $
for suitable $\lambda\in\R$,
we may assume  $V_{2,1}=0$. All of these changes do not affect the energy density.\\
\\
Hence, $V,W$ are affine on the intersection of $[abcd]_h$ with the upper and
lower half-plane respectively, and there exist  $A_1,\dots,A_8\in\R$ such
that
\[
\begin{split}
  V(x_1,x_2)=&\begin{cases}(A_1x_1+A_2x_2,A_3x_2) &\text{ if }x_2\geq 0\\
    (A_1 x_1+A_4 x_2,A_5x_2)&\text{ if }x_2< 0\end{cases}\\
  W(x_1,x_2)=&\begin{cases} A_6x_1+A_7x_2&\text{ if }x_2\geq 0\\
    A_6 x_1+A_8x_2&\text{ if }x_2< 0\end{cases}\,.
\end{split}
\]
Additionally, from $DV+DV^T+DW\otimes DW=0$ we have the relations

\begin{equation}
2A_1+A_6^2=A_2+A_6A_7=A_4+A_6A_8=2A_3+A_7^2=2A_5+A_8^2=0\,.\label{eq:4}
\end{equation}
Now let $\tau\in(0,1]$ be the largest number such that
$b$ and $d$ are outside the rhombus with $ac$ as diagonal, and sides
with slope $\pm\tau$. This condition means that (assuming $d_2<0<b_2$ to fix ideas)
one has
$b_2\ge \tau b_1$, $b_2\ge \tau (l-b_1)$, 
$d_2\le -\tau d_1$, $d_2\le -\tau (l-d_1)$. 

Since we have assumed that $h \in (0,l/8)$, we may apply Lemma 2.4 in \cite{Maggi08} to 
  choose $f\in
C^\infty([0,l];[0,\infty))$  
such that $f(x)=f(l-x)$ for $x\in[0,l]$, and such that $f$ and its
first two derivatives are controlled on $(0,l/2)$ by
\[
f_0(x)=\tau h^{1/3}(h+x)^{2/3}-\tau h\,,
\]
in the sense that there exists $C>0$ that is independent of $h$ and  $\tau$ such
that on $(0,l/2)$, we have
\begin{equation}
C^{-1}|f_0|\leq |f|\leq |f_0|,\quad |f'|\leq C|f_0'|,\quad|f''|\leq C|f_0''|\,.\label{eq:5}
\end{equation}

Now we choose  $\gamma_2,\gamma_3\in C^\infty([-1,1])$, that only depend on $A_3,A_5,A_7,A_8$,
with the following properties:
\begin{equation}
\begin{split}
  \gamma_2(t)=&\begin{cases} A_3 t &\text{ if }t\geq\frac34\\
    A_5 t&\text{ if }t\leq-\frac34\end{cases}\\
  \gamma_3(t)=&\begin{cases} A_7 t &\text{ if }t\geq\frac34\\
    A_8 t&\text{ if }t\leq-\frac34\end{cases}\\
2\gamma_2'+\gamma_3'^2=&0\,.
\end{split}\label{eq:3}
\end{equation}
Such a choice of $\gamma_2, \gamma_3$ is possible by Lemma \ref{lem:gamma} below.
For $t\in [-1,1]$, we define the following auxiliary objects,
\[
\eta_i(t)=\gamma_i(t)-t\gamma_i'(t)\,\quad\text{ for }i=2,3\,,
\]
and 
\[
\begin{split}
  \omega(t)=&\int_{-1}^t \eta_2(s)+\gamma_3'(s)\eta_3(s)\d s\\
  \xi(t)=&t\omega'(t)-\omega(t)\,.
\end{split}
\]
In particular, $\eta=\omega=\xi=0$ for $|s|\ge 3/4$. Furthermore, we compute
\begin{equation}
\begin{split}
  \omega(1)=&\int_{-1}^1
  \gamma_2(s)-s\gamma_2'(s)+\gamma_3(s)\gamma_3'(s)-s\gamma_3'^2(s)\d s\\
  =& \int_{-1}^1
  \gamma_2(s)+s\gamma_2'(s)+\gamma_3(s)\gamma_3'(s)\d s\\
   =&\gamma_2(1)+\gamma_2(-1)+\frac12\left(\gamma_3^2(1)-\gamma_3^2(-1)\right)\\
  =&A_3-A_5+\frac12(A_7^2-A_8^2)\\
  =&0\,.
\end{split}\label{eq:1}
\end{equation}
As a further auxiliary function, we define
\[
\beta(x_1,x_2)=\begin{cases}-f(x_1)f'(x_1)\omega(x_2/f(x_1))&\text{ if
  }-f(x_1)\leq x_2\leq f(x_1)\\ 0&\text{ else.}\end{cases}
\]
We are ready to define $(V^h,W^h):[abcd]\to \R^2\times \R$. We set

\begin{equation}
\begin{array}{rll}
  V^h_1(x_1,x_2)=& \left\{\begin{array}{l} 
A_1x_1-A_6
    f(x_1)\gamma_3\left(\frac{x_2}{f(x_1)}\right)+\beta(x_1,x_2)\\ V_1(x_1,x_2)\end{array}\right.&\begin{array}{l} \text{ for }-f(x_1)\leq x_2\leq
    f(x_1)\\ \text{ else}\end{array}\\
V^h_2(x_1,x_2)=& \left\{\begin{array}{l} f(x_1)\gamma_2\left(\frac{x_2}{f(x_1)}\right)\\ V_2(x_1,x_2)\end{array}\right.&\begin{array}{l} \text{ for
  }-f(x_1)\leq x_2\leq   f(x_1)\\\text{ else}\end{array}\\
  W^h(x_1,x_2)=&\left\{\begin{array}{l} A_6 x_1 +f(x_1)\gamma_3\left(\frac{x_2}{f(x_1)}\right)\\ W(x_1,x_2)\end{array}\right.&\begin{array}{l} \text{ for
    }-f(x_1)\leq x_2\leq   f(x_1)\\ \text{ else.}\end{array}
\end{array}\label{eq:8}
\end{equation}
Since $f(x)\le f_0(x)\le \tau x$, by the choice of $\tau$ the required boundary data are satisfied.
By \eqref{eq:4}, \eqref{eq:5}, \eqref{eq:3}, and \eqref{eq:1}, $V^h$ and $W^h$ are indeed in
$C^2([abcd]\setminus \{a,c\};\R^2)$ and $C^2([abcd]\setminus \{a,c\})$
respectively, and  $|DV^h|+|DW^h|\leq C$.
We now compute the components of $DV^h+(DV^h)^T+DW^h\otimes DW^h$. For
$|x_2|>f(x_1)$, they all identically vanish, while for $|x_2|\leq f(x_1)$, we have
\[
\begin{split}
  2V^h_{1,1}+(W^h_{,1})^2=&2(A_1-A_6f'\eta_3+\beta_{,1})+(A_6+f'\eta_3)^2\\
  =&2(f'^2\xi-ff''\omega)+f'^2\eta_3^2\\
  V^h_{1,2}+V^h_{2,1}+W^h_{,1}W^h_{,2}=&
  -A_6\gamma_3'-f'(\eta_2+\gamma_3'\eta_3)+f'\eta_2+(A_6+f'\eta_3)\gamma_3'\\
  =&0\\
  2V^h_{2,2}+(W^h_{,2})^2=&2\gamma_2'+\gamma_3'^2\\
 =&0\,.
\end{split}
\]
Here, the functions
$\omega,\xi,\gamma_i,\eta_i$, $i=2,3$, and their derivatives have  been
evaluated at $x_2/f(x_1)$.

We come to the estimate of  the membrane energy density. We introduce  the notation $h'=h/\sqrt{2}$ and note that,
since $\tau\le 1$,
\[
\begin{split}
  \{(x_1,x_2):0\leq x_1\leq& l, |x_2|\leq f(x_1)\}\setminus (B_h(a)\cup
  B_h(c))\\
&\subset \{(x_1,x_2):h'\leq x_1\leq l-h', |x_2|\leq f(x_1)\}.
\end{split}
\]
Thus, we may estimate 
\[
\begin{split}
  \int_{[abcd]_h}&|DV^h+(DV^h)^T+DW^h\otimes DW^h|^2\d x\\
  \lesssim& \int_{h'}^{l-h'}\d x_1
  \int_{-f(x_1)}^{f(x_1)}\d x_2
  \Big(f'^4(x_1)\xi^2(x_2/f)+f^2(x_1)f''^2(x_1)\omega^2(x_2/f)\\
&+f'^4(x_1)\eta_3^4(x_2/f)\Big)\,,
\end{split}
\]
where we also used the elementary inequality $(z_1+z_2+z_3)^2\lesssim
z_1^2+z_2^2+z_3^2$. 
By the change of variables $t=x_2/f(x_1)$, we obtain
\[\begin{split}
\int_{[abcd]_h}&|DV^h+(DV^h)^T+DW^h\otimes DW^h|^2\d x\\
  \lesssim
   &\int_{h'}^{l-h'}\d x_1 \Bigg[\left(f'^4(x_1)f(x_1)\int_{-1}^1\xi^2(t)\d
    t\right)+\left(f^2(x_1)f''^2(x_1)f(x_1)\int_{-1}^1\omega^2(t)\d
    t\right)\\
&+ \left(f'(x_1)^4f(x_1)\int_{-1}^1\eta_3^4(t)\d t\right)\Bigg]\\
\lesssim &\int_{h'}^{l-h'} f(f'^4+f^2f''^2)\d x_1\\
  \lesssim &h^{5/3}\int_h^{l/2} \d x_1 (x_1+h)^{-2/3}\\
  \lesssim &h^{5/3}l^{1/3}\,.
\end{split}
\]
Here, we used the bounds
\eqref{eq:5} to obtain the next to last estimate.
We come to the computation of the bending energy. We have
\[
D^2W^h(x_1,x_2)=\left(\begin{array}{cc}f''\eta_3-f^{-2}\eta_3'x_2f'^2 &
    f^{-1}f'\eta_3'\\f^{-1}f'\eta_3' & f^{-1}\gamma_3''\end{array}\right)\,,
\]
where again the functions
$\gamma_i,\eta_i$, $i=2,3$, and their derivatives have  been
evaluated at $x_2/f(x_1)$. 

This yields
\[
\begin{split}
  \int_{[abcd]_h}&|D^2W^h|^2\d x\\
\lesssim & \int_{h'}^{l-h'} \d x_1\int_{-f(x_1)}^{f(x_1)}\d
  x_2
  \Big(f''^2(x_1)\eta_3^2(x_2/f)
+f^{-4}(x_1)f'^4(x_1)x_2^2\eta_3'^2(x_2/f)\\
&+f^{-2}(x_1)f'^2(x_1)\eta_3'^2(x_2/f)+f^{-2}(x_1)\gamma_3''^2(x_2/f)\Big)\,.
\end{split}
\]
Again using the change of variables $t=x_2/f(x_1)$, we have
\[
\begin{split}
 \int_{[abcd]_h}&|D^2W^h|^2\d x\\
\lesssim &\int_{h'}^{l-h'} \d x_1\Bigg[
\left(f''^2(x_1)f(x_1)\int_{-1}^1\eta_3^2(t)\d
  t\right)
+\left(f^{-1}(x_1)f'^4(x_1)\int_{-1}^1t^2\eta_3'^2(t)\d t\right)\\
&+\left(f^{-1}(x_1)f'^2(x_1)\int_{-1}^1\eta_3'^2(t)\d t\right)
+\left(f^{-1}(x_1)\int_{-1}^1\gamma_3''^2(t)\d t\right)\Bigg]\,.
\end{split}
\]
Using the bounds \eqref{eq:5}, we obtain
\[
\begin{split}
  \int_{[abcd]_h}|D^2W^h|^2\d x\lesssim & \int_{h}^{l/2}\left(\frac{h}{(x_1+h)^{2}}+\frac{h^{1/3}}{(x_1+h)^{4/3}}+\frac{h^{-1/3}}{(x_1+h)^{2/3}}\right)\d x_1\\
    \lesssim & 1+ h^{-1/3} l^{1/3}\,.
  \end{split}
\]
Combining the estimates for membrane and bending energy, we obtain
\[
\bar E_h[V^h,W^h,[abcd]_h]\lesssim h^{5/3}l^{1/3}+h^2 \lesssim h^{5/3}l^{1/3}\,.
\]
This proves the lemma.
\end{proof}
\begin{lemma}
  \label{lem:gamma}
Given $A_7,A_8\in\R$, there exist $\gamma_2,\gamma_3\in C^\infty([-1,1])$ with
the following properties:
\[
\begin{split}
  \gamma_2(t)=&\begin{cases}-\frac12 A_7^2 t & \text{ if }t\geq \frac34\\
    -\frac12 A_8^2 t & \text{ if }t\leq -\frac34\end{cases}\\
  \gamma_3(t)=&\begin{cases}A_7 t & \text{ if }t\geq \frac34\\
    A_8 t & \text{ if }t\leq -\frac34\end{cases}\\
2\gamma_2'+\gamma_3'^2=&0\,.
\end{split}
\]
\end{lemma}
\begin{proof}
 This is very similar to Lemma 2.2 in \cite{Maggi08}, and we will be brief. Define
 $\tilde \gamma_3:[-1,1]\to\R$ by
\[
\begin{split}
\tilde  \gamma_3(t)=&\begin{cases}A_7 t & \text{ if }t\geq 0\\
    A_8 t & \text{ if }t<0\end{cases}\,.
\end{split}
\]
Now let $\varphi\in C^\infty_0(\R)$ be a  mollifier with $\int_\R\varphi\,dx=1$,
$\supp\varphi\subset (-1/8,1/8)$. We let  $\gamma_{3,\lambda}=\varphi*\tilde
\gamma_3+\lambda \varphi$. By continuity there is a  choice of $\lambda_0$ such that 
$\int_{-1}^{1}\gamma_{3,\lambda_0}'^2(s)\d s=A_7^2+A_8^2$. Then we set
$\gamma_3=\gamma_{3,\lambda_0}$ and
\[
\gamma_2(t)=\frac12 A_8^2-\frac12\int_{-1}^t\gamma_3'^2(s)\d s\,.
\] 
This choice of
$\gamma_2,\gamma_3$ fulfills all the required properties.
\end{proof}

\appendix
\section{}
\label{sec:appendix}

\subsection{The space of radially symmetric configurations}

The following lemma clarifies the definition of the space $A_\delta$ from (\ref{eq:Adelta}). Recall the change of variables $(u,w)\leftrightarrow (U,W)$ given in (\ref{eq:changeofvar}), and define the set
\[
X=\{(u,w):(U,W)\in W^{1,2}(B_{1};\mathbb{R}^{2})\times W^{2,2}(B_{1})\}.
\]

\begin{lemma}\label{lem:radsymmspace}
We have the equivalence
\[
\begin{split}X=\{(u,w):&u\in L^{2}((0,1),r^{-1}dr),u'\in L^{2}((0,1),rdr),\\
&w\in L^{2}((0,1),rdr),w'\in L^{2}((0,1),r^{-1}dr),w''\in L^{2}((0,1),rdr)\}
\end{split}
\]
where the derivatives are understood in the sense of distributions
on $(0,1)$. Moreover, if $(u,w)\in X$ then $w\in W^{1,2}((0,1))$.
In particular, $w$ has well-defined trace at $r=0,1$. Finally, $w'$
has a unique continuous representative on $(0,1)$, which extends
continuously to $(0,1]$.\end{lemma}
\begin{proof}
First, we note the elementary fact that $W^{1,2}(B_{1})=W^{1,2}(B_{1}\backslash\{0\})$.
Now, we compute for $x\in B_{1}\backslash\{0\}$ that 
\begin{align*}
DU(x) & =\frac{1}{2}(u'(|x|)-1)\hat{x}\otimes\hat{x}+\frac{1}{2}(\frac{u(|x|)}{|x|}-1)\hat{x}^{\perp}\otimes\hat{x}^{\perp}\\
DW(x) & =w'(|x|)\hat{x}\\
D^{2}W(x) & =w''(|x|)\hat{x}\otimes\hat{x}+\frac{w'(|x|)}{|x|}\hat{x}^{\perp}\otimes\hat{x}^{\perp}.
\end{align*}
Thus, $U\in W^{1,2}(B_{1};\mathbb{R}^{2})$ if and only if $u\in L^{2}((0,1),rdr)\cap L^{2}((0,1),r^{-1}dr)$
and $u'\in L^{2}((0,1),rdr)$. Also, $W\in W^{2,2}(B_{1})$ if and
only if $w\in L^{2}((0,1),rdr)$, $w'\in L^{2}((0,1),rdr)\cap L^{2}((0,1),r^{-1}dr)$,
and $w''\in L^{2}((0,1),rdr)$. This proves the required equivalence. 

The remaining assertions follow easily. To show that
$w\in W^{1,2}((0,1))$, note that by the equivalence, $w\in L_{loc}^{2}((0,1),dr)$
and $w'\in L^{2}((0,1),dr)$ so that $w\in L^{2}((0,1),dr)$. To show
that $w'$ has the asserted representative, note that $w'\in W^{1,2}((\delta,1))$
for all $\delta\in(0,1)$ so that the desired result follows from
a standard Sobolev embedding. 
\end{proof}

\subsection{A Gagliardo-Nirenberg type inequality}
We used the following interpolation inequality in the proof of the
lower bounds. In a slight abuse of notation, we will refer to $f\in W^{2,2}(\R)$ as both a Sobolev function and  as its unique continuous representative (the existence of which is ensured by a standard Sobolev embedding).
\begin{lemma}
\label{lem:interp_osc}Let $I\subset\R$ be a bounded interval and let $f\in W^{2,2}\left(I\right)$.
Then,
\begin{equation}
\osc{I}\left\{ f\right\} \lesssim\norm{f}_{L^{2}\left(I\right)}^{3/4}\norm{f''}_{L^{2}\left(I\right)}^{1/4}+\frac{1}{\abs{I}^{1/2}}\norm{f}_{L^{2}\left(I\right)}.\label{eq:interp_osc}
\end{equation}
\end{lemma}
\begin{proof}
By density we can assume $f\in C^{\infty}\left(I\right)$. Also, we
can assume that $\fint_{I}f\, dx=0$, and hence that there exists $x_{0}\in I$
such that $f\left(x_{0}\right)=0$. Let $t\in I$. Then by the fundamental
theorem of calculus,
\[
\frac{1}{2}f^{2}(t)=\int_{x_{0}}^{t}ff'\, dx\leq\norm{f}_{L^{2}\left(I\right)}\norm{f'}_{L^{2}\left(I\right)}.
\]
Using the Gagliardo-Nirenberg interpolation inequality
\[
\norm{f'}_{L^{2}\left(I\right)}\lesssim\norm{f}_{L^{2}\left(I\right)}^{1/2}\norm{f''}_{L^{2}\left(I\right)}^{1/2}+\frac{1}{I}\norm{f}_{L^{2}\left(I\right)},
\]
we conclude that 
\[
\abs{f\left(t\right)}\lesssim\norm{f}_{L^{2}\left(I\right)}^{3/4}\norm{f''}_{L^{2}\left(I\right)}^{1/4}+\frac{1}{\abs{I}^{1/2}}\norm{f}_{L^{2}\left(I\right)}
\]
for all $t\in I$. Letting $t,t'\in I$ and applying the triangle
inequality along with this estimate twice proves that
\[
\abs{f\left(t\right)-f\left(t'\right)}\lesssim\norm{f}_{L^{2}\left(I\right)}^{3/4}\norm{f''}_{L^{2}\left(I\right)}^{1/4}+\frac{1}{\abs{I}^{1/2}}\norm{f}_{L^{2}(I)}.
\]
Since $t,t'$ were arbitrary, the result follows.\end{proof}

\section*{Acknowledgments} 
This work was partially supported 
by the Deutsche Forschungsgemeinschaft through the Sonderforschungsbereich 1060 
{\sl ``The mathematics of emergent effects''}. I.T.\ was supported by a National Science Foundation Graduate Research Fellowship DGE-0813964; and National Science Foundation grants OISE-0967140 and DMS-1311833.

\bibliographystyle{amsplain-initials}
\bibliography{regular}

\end{document}